\documentclass[reqno,a4paper,11pt]{amsart}
\usepackage[utf8x]{inputenc}
\usepackage[T1]{fontenc}
\usepackage[english]{babel}
\usepackage[top=2.1cm, bottom=2.1cm, left=1.7cm,right=1.7cm]{geometry} 
\usepackage{amscd,amsfonts,amsmath,amssymb,amstext,amsthm}
\usepackage{enumitem}
\setlist[enumerate]{itemsep=0mm}
\setlist[itemize]{itemsep=0mm}
\usepackage{graphicx}
\usepackage{float}
\usepackage[bookmarksnumbered=true,pdfstartview=FitH,linkcolor=blue,citecolor=blue]{hyperref}   
\hypersetup{colorlinks=true,breaklinks=true,urlcolor=blue,linkcolor=blue}  
\usepackage{color}

\newcommand{\fonction}[5]{#1\colon\left\{\begin{array}{ccc} #2 &  \rightarrow & #3 \\  #4 & \mapsto & #5  \end{array}\right.}      

\usepackage{tikz}
\usetikzlibrary{decorations.markings,matrix,arrows,calc,shapes,decorations.pathreplacing}

\newtheorem{thm-intro}{Theorem}
\newtheorem{thm}{Theorem}[section]

\newtheorem{prop}[thm]{Proposition}

\theoremstyle{definition}
\newtheorem{defn}[thm]{Definition}
\newtheorem{defs}[thm]{Definitions}
\newtheorem{rmk}[thm]{Remark}

\newcommand{\elem}{element }

\renewcommand{\Im}{\mathrm{Im\,}}
\newcommand{\Ker}{\mathrm{Ker\,}}
\newcommand{\Pic}{\mathrm{Pic\,}}

\newcommand{\Spec}{\mathrm{Spec\,}}

\newcommand{\PGL}{\mathrm{PGL\,}}
\DeclareMathOperator{\Id}{Id}
\newcommand{\N}{\mathbb N}
\newcommand{\Z}{\mathbb Z}
\newcommand{\Q}{\mathbb Q}
\newcommand{\R}{\mathbb R}
\newcommand{\C}{\mathbb C} 

\renewcommand{\phi}{\varphi}
\renewcommand{\bar}{\overline}
\renewcommand{\P}{\mathbb P}

\newcommand{\Aut}{\mathrm{Aut\,}}

\renewcommand{\PGL}{\mathrm{PGL}}
\renewcommand{\O}{\mathrm{O}}
\newcommand{\intl}{[\![}
\newcommand{\intr}{]\!]}

\renewcommand{\phi}{\varphi}

\renewcommand{\tilde}{\widetilde}
\newcommand{\dbar}[1]{\overline{\rule[-.3\baselineskip]{0pt}{0.8\baselineskip} #1}}
\newcommand{\dsum}{\displaystyle\sum}

\begin{document}
\title[Real structures on rational surfaces]{Real structures on rational surfaces \\and automorphisms acting trivially on Picard groups}
\author{Mohamed Benzerga}
\address{Universit\'e d'Angers, \textsc{Larema}, UMR CNRS 6093, 2, boulevard Lavoisier, 49045 Angers cedex 01, France}
\email{mohamed.benzerga@univ-angers.fr}
\noindent

\bibliographystyle{amsalpha} 
\begin{abstract}
In this article, we prove that any complex smooth rational surface $X$ which has no automorphism of positive entropy has a finite number of real forms  (this is especially the case if $X$ cannot be obtained by blowing up $\P^2_{\C}$ at $r\geq 10$ points). In particular, we prove that the group $\Aut^{\#}X$ of complex automorphisms of $X$ which act trivially on the Picard group of $X$ is a linear algebraic group defined over $\R$.
\end{abstract}
\renewcommand{\subjclassname}
{\textup{2010} Mathematics Subject Classification}

\maketitle

\section*{Introduction}\label{intro}
In real algebraic geometry, one is interested in the problem of the classification of real varieties which become isomorphic after complexification. For example, the complexification of $\P^1_{\R}$ is of course $\P^1_{\C}$, and the complexification of the real conic $X_0$ of equation $x^2+y^2+z^2 = 0$ in $\P^2_{\R}$ (i.e. the complex conic of $\P^2_{\C}$ defined by the same equation) is also isomorphic to $\P^1_{\C}$ because the conic is smooth rational over $\C$. However, $X_0$ is not isomorphic to $\P^1_{\R}$ as a real variety : indeed, the set of real points of $X_0$ is empty whereas $\P^1_{\R}$ has real points. One says that $X_0$ and $\P^1_{\R}$ are two distinct \textit{real forms} of $\P^1_{\C}$. We now define precisely the concepts involved in order to present the problem with which this paper deals.\\
\indent A \textit{real structure} on a complex algebraic variety $X$ is an antiregular\footnote{This is equivalent to "antiholomorphic" if $X$ is projective.} involution $\sigma$ on $X$ ; the antiregularity condition means that the following diagram commutes :
\begin{center}
\begin{tikzpicture}[scale=2][description/.style={fill=white,inner sep=2pt}]
  \matrix (m) [matrix of math nodes, row sep=3em,
  column sep=4.5em, text height=1.5ex, text depth=0.25ex]
  {X& X \\
  \Spec\C&\Spec\C \\ };
  \path[->,font=\scriptsize]
  (m-1-1) edge node[auto] {$\sigma$} (m-1-2)
  edge node[auto] {} (m-2-1)
  (m-1-2) edge node[auto] {} (m-2-2)
  (m-2-1) edge node[auto] {$\Spec(z\mapsto \bar{z})$} (m-2-2);
\end{tikzpicture}
\end{center}
The data of a complex quasiprojective variety $X$ (as a scheme over $\C$) and of a real structure $\sigma$ on it corresponds to the data of a scheme $X_0$ over $\R$ : if $X:=X_0\times_{\Spec\R}\Spec\C$ is the complexification of $X_0$, then $X_0$ corresponds to the natural real structure $\Id_{X_0}\times\Spec(z\mapsto \bar{z})$. Conversely, if $(X,\sigma)$ is given, then $X_0 = X/\langle \sigma\rangle$ is a scheme over $\R$ such that $X_0\times_{\Spec \R}\Spec\C \simeq X$ : it is called a \textit{real form} of $X$ (see \cite[II. Ex.4.7]{hartshorne}, \cite[I.1.4]{silhol}, \cite[Prop.2.6]{borel-serre}).\\
\indent Two real structures $\sigma$ and $\sigma'$ on $X$ are \textit{equivalent} if there exists a $\C$-automorphism $\phi$ of $X$ such that $\sigma'=\phi\sigma\phi^{-1}$. Such an automorphism $\phi$ is equivariant for the two structures, thus it corresponds to an $\R$-isomorphism between the two $\R$-schemes $X_0=X/\langle \sigma\rangle$ and $X'_0 = X/\langle\sigma'\rangle$. Conversely, an $\R$-isomorphism between two $\R$-schemes $X_0$ and $X'_0$ corresponds to a $\C$-isomorphism between their complexifications which is equivariant with respect to the natural real structures ; thus, the equivalence classes of real structures on $X$ correspond to the isomorphism classes of real forms of $X$. Moreover, if $G=\langle \sigma\rangle$, then $H^1(G,\Aut_{\C} X)$ is the set of equivalence classes of real structures on $X$, where $G$ acts on $\Aut_{\C} X$ by conjugation (cf. \cite[2.6]{borel-serre}). The aim of this article is to give a partial answer to the following open problem :\\
\indent \textbf{Problem.}\,\,\textit{Let $X$ be a rational surface. Does $X$ have a finite number of real forms ? Namely, is there a finite number of equivalence classes of real structures on $X$ ?} \\
\indent This question was asked for algebraic surfaces in general by Kharlamov in \cite{kharlamov}. It has positive answer for minimal rational surfaces, cf.\cite{kharlamov}, \cite[III.6.11.7]{dik}, also for minimal algebraic surfaces of nonnegative Kodaira dimension, cf. \cite[Appendix D]{dik}, and in the case of Del Pezzo surfaces, cf.\cite{russo} (see also \cite{kollar} which uses minimal model program for schemes over $\R$). Let us also note that the finiteness of real structures is known for projective spaces, cf. \cite[1.1]{russo}, for abelian varieties, cf. \cite[3.5,6.1]{borel-serre} and \cite[Prop.7]{silhol82}, and for varieties of general type (since their automorphism groups are finite, see for example \cite{matsumura-bir} or \cite{autos-generaltype}).\footnote{So that one can obtain finiteness for smooth projective curves.} Finally, it should be noted that, to our knowledge, there is no example of a variety having an infinite number of non-equivalent real forms.\\
\indent In view of the above, we will need to understand automorphism groups of rational surfaces ; thus, we briefly recall classical results about automorphisms of rational surfaces (in the following, every rational surface is supposed to be \textit{complex and smooth}). If $X$ is a rational surface, then $\Pic X$ is a free abelian group of finite rank $\rho(X)$ and the intersection form on $X$ induces a symmetric bilinear form of signature $(1,\rho(X)-1)$ on $\Pic X$ by Hodge Index Theorem. $\Aut X$ acts isometrically on $\Pic X$ by pulling-back divisor classes and if $p : \Aut X\to\O(\Pic X)$ is the morphism corresponding to this action, we denote by $\Aut^{\#}X$ the kernel of $p$ and by $\Aut^*X$ the image of $p$ (as in \cite{harbourne}) and thus we have an exact sequence :
$$1 \longrightarrow \Aut^{\#}X \longrightarrow \Aut X \longrightarrow \Aut^*X \longrightarrow 1$$
\indent These two groups are very different in nature : $\Aut^{\#}X$ is a linear algebraic group whereas  $\Aut^*X$ is a discrete group of isometries of the lattice $\Pic X$. In particular, if $X$ is a basic rational surface\footnote{\,In our work, we will have to distinguish two kinds of rational surfaces : a rational surface is called \textit{basic} if it dominates $\P^2$ and non-basic otherwise.}, then $\Aut^*X$ is included in a Weyl group $W_X$ (see \cite[VI.Lemma 3]{dolgachev-ortland} or \cite[Corollary p.283]{nagata-rat-surfaces-II}) which is finite if, and only if, the number $r$ of points which are blown-up over $\P^2$ to obtain $X$ satisfies $r\leq 8$. Finally, thanks to results of Gromov \cite{gromov-entropy} and Yomdin \cite{yomdin}, if $\phi\in\Aut X$, we may define the \textit{topological entropy} $h(\phi)$ of $\phi$ as the logarithm of the spectral radius of the automorphism $\phi^* $ of $\Pic X\otimes_{\Z}\R$ induced by $\phi$. As we will see, the presence of automorphisms of positive entropy in $\Aut X$ indicates in some sense how big is the group $\Aut^*X$.\\
\indent In fact, our main result is the following theorem : 
\begin{thm-intro}\label{main} If a rational surface $X$ has an infinite number of non-equivalent real structures, then $X$ is a blowing up of $\P^2$ at $r\geq 10$ points and has at least one automorphism of positive entropy. \end{thm-intro}
Let us note that this result contains the case of del Pezzo surfaces and minimal rational ones, but it also contains the different cases of blow-ups of points in special positions (collinear points, conical sextuplets, infinitely near points...) ; however, in section \ref{proof-main}, we derive this result in a synthetic manner. Before this, we will need to study $\Aut^{\#}X$ in section \ref{aut-diese} : we prove that a real structure $\sigma$ on $X$ induces a real structure on $\Aut^{\#}X$. Finally, in section \ref{other}, we show some other finiteness results, which are not included in Theorem \ref{main} : in particular, we prove that if a rational surface $X$ has an infinite number of non-equivalent real structures, then $\Aut^*X$ must contain a non-abelian free group.

\section{Real structure on $\Aut^{\#}X$}\label{aut-diese}
We will use the following convention (cf. \cite[I.4]{silhol}) :
\begin{defn} Let $X$ be a smooth projective complex variety, $\phi$ be an automorphism of $X$ and $\sigma$ be a real structure on $X$.
\begin{enumerate}
\item If $f : X \dashrightarrow \C$ is a rational function on $X$, then we set $\phi^*f := f\circ\phi$ and $\sigma^*f := \bar{f\circ\sigma}$.
\item If $D=\{(U_i,f_i)\}$ is a Cartier divisor, then we set $\phi^*D := \{(\phi(U_i),\phi^*f_i)\}$ and $\sigma^*D := \{(\sigma(U_i),\sigma^*f_i)\}$. 
\end{enumerate}
\end{defn}
\begin{thm}\label{autdiese} Let $X$ be a smooth irreducible projective complex variety and $\sigma$ be a real structure on $X$.\\
$\Aut^{\#}X$ is a complex linear algebraic group\footnote{This result is probably well-known but we prove it in the course of proving the second part of this Theorem.} and the map $\fonction{\tilde{\sigma}}{\Aut^{\#}X}{\Aut^{\#}X}{\phi}{\sigma\phi\sigma^{-1}}$ defines a real structure on $\Aut^{\#}X$. \end{thm}

\begin{proof}[Sketch of proof] The idea of the proof is the following : we use an embedding of $X$ in a projective space $\P^n$ and we define a real structure $\sigma_1$ on $\P^n$ for which this embedding is equivariant. Then we show that the elements of $\Aut^{\#}X$ can be extended to automorphisms of $\P^n$, realizing $\Aut^{\#}X$ as an algebraic subgroup of $\PGL_{n+1}(\C)$. Finally, we are reduced to prove that $\fonction{\tilde{\sigma_1}}{\Aut\P^n}{\Aut\P^n}{\phi_1}{\sigma_1\phi_1\sigma_1^{-1}}$ defines a real structure on $\Aut\P^n$ which stabilizes the image of $\Aut^{\#}X$ and this is finally a computation in homogeneous coordinates in $\P^n$.
\begin{itemize}[label=\textbullet,font=\small]
\item \textbf{Step 1 : the embedding}\\
Since $X$ is projective, we can find an ample divisor $D_0$ on $X$ : by Nakai-Moishezon criterion, $D_1 := D_0+\sigma^*D_0$ is also ample on $X$ and there exists $p\in\N$ such that $D := pD_1$ is very ample on $X$ and satisfies $\sigma^*D = D$.\\
Now, let $(s_0,\dots,s_n)$ be a basis of the $\C$-vector space 
$$H^0(D) := \{ f\in\mathcal K(X)^*|\, D+(f)\geq 0\}\cup \{0\}$$
We consider the embedding defined by this basis :
$$\fonction{\Phi}{X}{\P^n}{x}{[s_0(x) : \dots : s_n(x)]}$$
\item \textbf{Step 2 : the real structure}\\
Note that $\forall j\in \intl 0 ; n\intr$, $\sigma^*s_j\in H^0(D)$ since $\sigma^*D=D$ : thus, there exist some complex numbers $\sigma_{j0},\dots,\sigma_{jn}$ such that $\displaystyle \sigma^*s_j = \dbar{s_j\circ\sigma} = \sum_{i=0}^n\sigma_{ji}s_i$. We now define $\sigma_1$ by :
$$\fonction{\sigma_1}{\P^n}{\P^n}{\left[x_0:\dots:x_n\right]}{\left[\dsum_{i=0}^n \dbar{\sigma_{0i}x_i} : \dots : \dsum_{i=0}^n \dbar{\sigma_{ni}x_i}\right]}$$
One can verify that $\sigma_1$ is a real structure on $\P^n$ and that $\Phi$ is equivariant for $\sigma$ and $\sigma_1$.
\item\textbf{Step 3 : building an automorphism of $\P^n$ from an element of $\Aut^{\#}X$}\\
If $\phi\in\Aut^{\#}X$, then $\phi^*D\sim D$ thus there exists $f_{\phi}\in\mathcal K(X)^*$ such that $\phi^*D = D + (f_{\phi})$. Now, note that $\forall j\in \intl 0 ; n\intr$, $f_{\phi}\phi^*s_j\in H^0(D)$ since $\phi^*s_j\in H^0(\phi^*D)$ : thus, there exist some complex numbers $a_{j0},\dots,a_{jn}$ such that $\displaystyle f_{\phi}\phi^*s_j = \sum_{i=0}^n a_{ji}s_i$. We now define $\phi_1\in\Aut \P^n$ by :
$$\fonction{\phi_1}{\P^n}{\P^n}{\left[x_0:\dots:x_n\right]}{\left[\dsum_{i=0}^n a_{0i}x_i : \dots : \dsum_{i=0}^n a_{ni}x_i\right]}$$
We have to verify that $\Phi\circ\phi = \phi_1\circ\Phi$. Let $\textrm{Ind}(f_{\phi})$ be the indeterminacy locus of $f_{\phi}$ and $U := \{ x\in X\setminus \textrm{Ind}(f_{\phi})|\;f_{\phi}(x)\neq 0\}$. For $x\in U$, one has :
\begin{eqnarray*}
\Phi(\phi(x)) &= &[s_0(\phi(x)) : \dots : s_n(\phi(x))]\\
&= &\left[f_{\phi}(x)\phi^*s_0(x) : \dots : f_{\phi}(x)\phi^*s_n(x)\right]\\
&= &\left[\sum_{i=0}^n a_{0i}s_i(x) : \dots : \sum_{i=0}^n a_{ni}s_i(x)\right] = \phi_1(\Phi(x))\\
\end{eqnarray*}
Hence $\Phi\circ\phi$ and $\phi_1\circ\Phi$ are two isomorphisms from $X$ to $\P^n$ which coincide on the dense open subset $U$ of $X$ : thus they coincide on all of $X$.
\item\textbf{Step 4 : the embedding of $\Aut^{\#}X$ into $\Aut\P^n$}\\
We define $\fonction{\beta}{\Aut^{\#}X}{\Aut\P^n}{\phi}{\phi_1}$ : $\beta$ is clearly injective (since the data of $\beta(\phi)$ determines $\beta(\phi)|_{\Phi(X)}$ and thus $\phi$) and a slightly tedious computation shows that $\beta$ is a group morphism. Hence $\Aut^{\#}X$ can be equipped with a structure of linear algebraic group over $\C$ for which $\beta$ is a morphism of algebraic groups.
\item\textbf{Step 5 : antiregularity of $\tilde{\sigma_1}$ and $\tilde{\sigma}$}\\
Finally, we define $\fonction{\tilde{\sigma}}{\Aut^{\#}X}{\Aut^{\#}X}{\phi}{\sigma\phi\sigma^{-1}}$ and $\fonction{\tilde{\sigma_1}}{\Aut\P^n}{\Aut\P^n}{\phi}{\sigma_1\phi\sigma_1^{-1}}$.\\
Now, a really tedious computation shows both that $\beta\tilde{\sigma} = \tilde{\sigma_1}\beta$ and that $\tilde{\sigma_1}$ is antiregular. Thus, we can conclude that $\tilde{\sigma} = \left(\beta|_{\Im\beta}\right)^{-1}\,\tilde{\sigma_1}\,\beta$ is antiregular ; as it is clearly an involution, we obtain the desired result.
\end{itemize} 
\end{proof}

\section{Proof of the Main Theorem}\label{proof-main}
As we have seen in the Introduction, if $X$ is a complex quasiprojective surface with a real structure $\sigma$ and $G=\langle\sigma\rangle$, then $H^1(G,\Aut X)$ is the set of equivalence classes of real structures on $X$, where $G$ acts on $\Aut X$ by conjugation (notice that this set is not defined if $X$ has no real structure, but in this case, our problem is trivially solved). Thus, we have to consider this cohomology set in order to prove Theorem \ref{main}.\\
\indent For the convenience of the reader, we now recall the fundamental definitions and results about non-abelian group cohomology of finite groups (one can generalize to profinite topological groups).
\begin{defs} Let $G$ be a finite group.
\begin{itemize}
\item A \textit{$G$-group} is a group $A$ on which $G$ acts by automorphisms. In other words, $A$ is equipped with an action of $G$ such that : $\forall \sigma\in G,\,\forall a,b\in A, \sigma.(ab) = (\sigma.a)(\sigma.b)$. If $A$ is abelian, one also calls it a \textit{$G$-module}.
\item We denote by $A^G$ or $H^0(G,A)$ the set of fixed points of this action.
\item A map $\fonction{a}{G}{A}{\sigma}{a_{\sigma}}$ is a \textit{cocycle} if :
$$\forall \sigma,\tau\in G,\,a_{\sigma\tau} = a_{\sigma}\sigma.a_{\tau}$$
We will denote by $Z^1(G,A)$ the set of cocycles.
\item Two cocycles $a$ and $b$ are \textit{equivalent} (we note $a\sim b$) if :
$$\exists \alpha\in A, \forall \sigma\in G,\,b_{\sigma} = \alpha^{-1}a_{\sigma}\sigma.\alpha$$
\item The \textit{first cohomology set of $G$ with coefficients in $A$} is $H^1(G,A) := Z^1(G,A)/\sim$. When $A$ is not abelian, it is only a \textit{pointed set} (the distinguished point being the class of the cocycle $a : \sigma \mapsto \Id$).
\item If $B$ is a $G$-group and $A$ a normal subgroup of $B$ stable by the action of $G$, then for any $b\in Z^1(G,B)$, one can define another action of $G$ on $A$ by :
$$\forall \sigma\in G,\;\forall a\in A,\;\sigma*a := b_{\sigma}\,\sigma.a\,b_{\sigma}^{-1}$$
The group $A$ endowed with this new action is denoted by $A_b$ (as a $G$-group) and we say that $A_b$ is obtained by \textit{twisting} the $G$-group $A$ by $b$.
\item If $(X,x)$ and $(Y,y)$ are two pointed sets, a \textit{morphism} between them is a map $f:X\to Y$ such that $f(x) = y$. 
\item An \textit{exact sequence of pointed sets} is a sequence of morphisms of pointed sets 
$$(X,x)\overset{f}{\longrightarrow} (Y,y) \overset{g}{\longrightarrow} (Z,z)$$
such that $\Im f = \Ker g := g^{-1}(\{z\})$. Note that $\Ker g = \{y\}$ does not imply that $g$ is injective.
\end{itemize} 
\end{defs}
\begin{rmk} By definition, every cocycle $a$ satisfies $a_{1_G} = \Id$ (where $1_G$ and $\Id$ are the neutral elements of $G$ and $A$ resp.). Moreover, since $G=\langle \sigma \rangle \simeq \Z/2\Z$ for us, one can easily verify that the data of a cocycle is equivalent to the data of an \elem $a_{\sigma}$ of $A$ such that $a_{\sigma}\sigma.a_{\sigma} = \Id$. \end{rmk}
There exists an exact sequence in non-abelian cohomology (see \cite[1.17,1.20]{borel-serre} or \cite[I.\S 5.5]{serre-cohgal-english}\footnote{\,In fact, the following proposition is slightly simpler than the original one, which replaces the hypothesis of the finiteness of $H^1(G,A_b)$ for every $b$ by the weaker hypothesis of the finiteness of $\Im(H^1(G,A_b)\to H^1(G,B_b)) \simeq (g^*)^{-1}(\{g^*(b)\})$ for every $b$.}):
\begin{prop}\label{exact} If $0\longrightarrow A \overset{f}{\longrightarrow} B \overset{g}{\longrightarrow} C \longrightarrow 0$ is an exact sequence of $G$-groups (in particular, $f$ and $g$ are $G$-equivariant), then we have the following exact sequence of pointed sets :
$$ 0 \longrightarrow A^G \longrightarrow B^G \longrightarrow C^G\longrightarrow H^1(G,A) \overset{f^*}{\longrightarrow} H^1(G,B) \overset{g^*}{\longrightarrow} H^1(G,C)$$ 
Moreover, if $H^1(G,C)$ is finite and if $H^1(G,A_b)$ is finite for all $b\in Z^1(G,B)$, then $H^1(G,B)$ is finite.\end{prop}
Since we want to apply this result to the exact sequence given in the Introduction, we have to check that this sequence is $G$-equivariant. But it suffices to remark that the $G$-action restricts to $\Aut^{\#}X$ (if $\phi\in\Aut^{\#}X$, then for every divisor $D$ on $X$, $(\sigma\phi\sigma^{-1})^*(D) \sim D$) and to define the $G$-action on $\Aut^*X$ by $\sigma.\phi^* := (\sigma\phi\sigma^{-1})^*$.

The exact sequence of cohomology sets and the following Theorem are our main tools for proving Theorem \ref{main} :\newpage
\begin{thm}\label{coho} Let $G$ be a finite group.
\begin{enumerate}
\item \emph{\cite[6.2]{borel-serre}} If $A$ is a linear algebraic group defined over $\R$ (i.e. a complex linear algebraic group equipped with a real structure), then $H^1(\mathrm{Gal}(\C/\R),A)$ is finite for the "natural action" (corresponding to the real structure of $A$).
\item If $A$ contains a subgroup of finite index isomorphic to $\Z^k$ for some $k\in\N$, then $H^1(G,A)$ is finite independently of the action of $G$ on $A$. (This is true in particular if $A$ is finite, i.e. $k=0$.)
\end{enumerate}
\end{thm}
\begin{proof} One can find a proof of the first point in \cite[6.2]{borel-serre}, so we only prove the second point, following an idea of Borel and Serre (\cite[3.8]{borel-serre}). Since $G$ acts on $A$ by automorphisms, we can form the semidirect product $A\rtimes G = \{(a,\sigma)|\,a\in A,\,\sigma\in G\}$ with the group law defined by 
$$(a,\sigma)(a',\sigma') := (a\,\sigma.a',\,\sigma\sigma').$$
If $a$ is a cocycle, then we can easily see that the map $\fonction{\tilde{a}}{G}{A\rtimes G}{\sigma}{(a_{\sigma},\sigma)}$ is a group morphism. Moreover, note that if $\alpha\in A$, $\sigma\in G$ and $a\in Z^1(G,A)$, then 
$$(\alpha^{-1},1)(a_{\sigma},\sigma)(\alpha,1) = (\alpha^{-1}a_{\sigma},\sigma)(\alpha,1) = (\alpha^{-1}a_{\sigma}\sigma.\alpha,\sigma).$$
Thus, two cocycles $a$ and $b$ are equivalent if and only if $\exists\alpha\in A$ such that $\forall \sigma\in G,\;(\alpha,1)\tilde{a}(\sigma)(\alpha^{-1},1) = \tilde{b}(\sigma)$ and this is equivalent to saying that the finite subgroups $\tilde{a}(G)$ and $\tilde{b}(G)$ of $A\rtimes G$ are conjugate by an element $(\alpha,1)$ of $A\rtimes G$.\\
\indent By assumption, $A$ contains a subgroup of finite index isomorphic to $\Z^k$ for some $k\in\N$ and $\Z^k$ is a solvable arithmetic group : thus, by \cite{platonov1} (see the Corollary of Theorem 1.3), $A$ is an arithmetic group. Since $G$ is finite, $A$ is of finite index in $A\rtimes G$ and thus $A\rtimes G$ has a finite number of conjugacy classes of finite subgroups by \cite[Th.1.4]{platonov1}. Now, this is not exactly what we want because we need to prove that there is a finite number of conjugacy classes \textit{only via elements of $A$}. But since $A$ is of finite index in $A\rtimes G$, the desired finiteness is a consequence of the finiteness of the number of conjugacy classes of finite subgroups.\\
\indent Thus, there exist $a_1,\dots,a_p\in Z^1(G,A)$ such that for every $a\in Z^1(G,A)$, there exists $i\in\intl 1;p\intr$ and $(\alpha,1)\in A\subseteq A\rtimes G$ such that $\tilde{a}(G) = (\alpha,1)\tilde{a_i}(G)(\alpha^{-1},1)$. Now, we can conclude that $H^1(G,A) = \{[a_1],\dots,[a_p]\}$ is finite.
\end{proof}

We now recall the statement of Theorem \ref{main}, before proving it :
\begin{thm}If a rational surface $X$ has an infinite number of non-equivalent real structures, then $X$ is a blowing up of $\P^2$ at $r\geq 10$ points and has at least one automorphism of positive entropy.\end{thm}

\begin{proof} We prove the contrapositive of the theorem. Recall that we use the following $G$-equivariant exact sequence 
$$1 \longrightarrow \Aut^{\#}X \longrightarrow \Aut X \longrightarrow \Aut^*X \longrightarrow 1$$
\indent If $X$ is a rational surface, we begin by showing that $H^1(G,(\Aut^{\#}X)_b)$ is finite for every $b\in Z^1(G,\Aut X)$ and then, we show that $H^1(G,\Aut^*X)$ is finite if $X$ is non-basic or if it is a basic surface with no automorphism of positive entropy ; by Proposition \ref{exact}, it suffices to conclude.
Now, remark that $G=\langle\sigma\rangle$ acts on $(\Aut^{\#}X)_b$ by $\sigma*\phi := b_{\sigma}(\sigma\phi\sigma^{-1})b_{\sigma}^{-1} = \sigma_b\phi\sigma_b^{-1}$ where $\sigma_b := b_{\sigma}\sigma$ is a new real structure on $X$ (here is the meaning of the torsion : it corresponds to a change of real structure). Thus, it is always the cohomology set corresponding to some real structure on $X$. Now, by Theorem \ref{autdiese}, the action of a real structure on $\Aut^{\#}X$ defines a real structure on this complex linear algebraic group. Thus, by \textit{(1)} of Theorem \ref{coho}, $H^1(G,(\Aut^{\#}X)_b)$ is finite for every $b\in Z^1(G,\Aut X)$.\\
\indent Now, let us remark that if $X$ is a non-basic rational surface (i.e. if $X$ is not a blow-up of $\P^2$), then $\Aut^*X$ is finite. Indeed, by \cite[Lemma 1.1]{harbourne}, if $\Aut^*X$ is infinite, then $X$ contains infinitely many exceptional curves. But, the Corollary 1.2 of \cite{harbourne} shows that in this case, $X$ dominates $\P^2$. Thus $H^1(G,\Aut^*X)$ is finite and every non-basic rational surface has a finite number of real forms.\\
\indent Finally, if $X$ is a blow-up of $\P^2$ having no automorphism of positive entropy, then by \cite[Th.3.13]{grivaux}, $A:=\Aut^*X\cap \mathrm{O}^{+}(\Pic X)\simeq \Z^k\rtimes H$, where $k\geq 0$ and $H$ is a finite group. But $A$ is of index at most two in $\Aut^*X$, so by \textit{(2)} of Theorem \ref{coho}, $H^1(G,\Aut^*X)$ is finite and $X$ has a finite number of real forms. In particular, if $X$ is obtained by blowing-up $r\leq 9$ points, then it has no automorphism of positive entropy by \cite[Prop.2.2]{diller}.
\end{proof}

\begin{rmk} In \cite{kharlamov}, finiteness is claimed for blow-ups at 10 points but, in fact, in a discussion with Kharlamov, we saw that his argument was only valid for unnodal Coble surfaces (this argument is used below in the proof of Proposition \ref{cr-special}).\end{rmk}

\section{Some other finiteness results}\label{other}
\subsection{Cremona special surfaces} Let us recall from \cite{cantat-dolgachev} the following definition :
\begin{defn} A point set $\mathcal P=\{p_1,\dots,p_r\}$ of $\P^2$ is called \textbf{Cremona special} if $r\geq 9$ and if the surface $X$ obtained by blowing up $\mathcal P$ is such that $\Aut^*X$ has finite index in the (infinite) Weyl group $W_X$. The surface $X$ is then also called Cremona special.\end{defn}
This condition of finiteness roughly expresses that $\Aut^*X$ is as large as possible. In \cite{cantat-dolgachev}, it is proved that a Cremona special complex rational surface is either an unnodal Halphen surface (with $r=9$) or an unnodal Coble surface (with $r=10$).
\begin{prop}\label{cr-special} If $X$ is a Cremona special rational surface, then $X$ has a finite number of real forms. \end{prop}
\begin{proof} We will prove that $\Aut^* X$ is of finite index in $\O(K_X^{\perp})$ : let us first explain why this is sufficient. If $\Aut^* X$ is of finite index in $\O(K_X^{\perp})$, then $\Aut^* X$ is an arithmetic $G$-group, since we can embed it in $\O(K_X^{\perp}\otimes_{\Z}\Q)$, which is a linear algebraic $G$-group defined over $\Q$. By \cite[3.8]{borel-serre}, we see that $H^1(G,\Aut^*X)$ is finite.\\
\indent Now, $\Aut^*X$ is of finite index in the Weyl group $W_X$ because $X$ is Cremona special, and in this case, $r=9$ or $r=10$. Thus, we have to explain why $W_X$ is of finite index in $\O(K_X^{\perp})$. If $r=9$, a classical result about the affine type Weyl group $W_9 = W(E_{2,3,6}) = W(\widetilde{E_8})$ (cf. \cite[p.112]{cossec-dolgachev} or \cite[Prop.7.5.9]{dolgachev-cag}) says that in this case
$$\O(K_X^{\perp})' = A(\widetilde{E_8}) \ltimes W_9$$
where $\O(K_X^{\perp})' := \{ \phi\in \O(K_X^{\perp})|\;\phi(K_X) = K_X\}$ is a subgroup of index 2 of $\O(K_X^{\perp})$ and $A(\widetilde{E_8})$ is the automorphism group of the Dynkin diagram of $\widetilde{E_8}$. Clearly, this diagram has no symmetries, so this group is trivial. Hence $W_9=\O(K_X^{\perp})'$ is of index 2 in $\O(K_X^{\perp})$.
If $r=10$, then $E_{10} := E_{2,3,7} \simeq K_X^{\perp}$ is a Nikulin lattice, that is, an even hyperbolic lattice isomorphic to $E_8\oplus H$, where $H$ is the "hyperbolic plane", i.e. a rank-two lattice equipped with a symmetric bilinear form of matrix $\begin{bmatrix} 0 & 1\\1& 0\end{bmatrix}$. The classification of Nikulin lattices (see \cite{dolgachev-bourbaki} or \cite[Ex.2.7]{dolgachev-reflections}) shows that $W_X$ is of finite index in $\O(K_X^{\perp})$.
\end{proof}
Remark that Theorem \ref{main} does not ensure the desired finiteness for general unnodal Coble surfaces, because they have automorphisms of positive entropy. (see \cite{deserti} : Déserti constructs ten elliptic fibrations on a Coble surface and claims that for a general Coble surface $X$, one obtains a subgroup of $\Aut X$ isomorphic to $(\Z^8)^{*10}$, hence by \cite[Th.3.13]{grivaux}, $\Aut X$ must contain automorphisms of positive entropy).

\subsection{Tits alternative} There is a kind of Tits alternative for automorphism groups of smooth projective surfaces :
\begin{thm} \cite[Th.1.6]{zhang} \label{tits}
Let $X$ be a smooth projective complex surface and $A$ a subgroup of $\Aut X$. If $A$ contains at least one automorphism of positive entropy, then $A$ satisfies exactly one of the following assertions :
\begin{itemize}
\item $A$ contains the non-abelian free group $\Z * \Z$; or
\item there is a normal subgroup $B$ of $A$ such that $|A / B| \leq 2$ and $B = \langle h_m \rangle \ltimes T$ with $h_m$ of positive entropy and $p(T)$ finite\footnote{Remember that $p$ is the natural morphism $\Aut X\to \O(\Pic X)$, which we have seen in the Introduction.}. Moreover, if $X$ is rational, then $T$ itself is finite.
\end{itemize}
\end{thm}
It may be surprising to note that we can prove a partial finiteness result in the first case :
\begin{prop}\label{h1-freeprod} Let $k\in\N$, $k\geq 2$. If $G=\Z/2\Z$, then $H^1(G,\Z^{*k})$ is finite independently of the action of $G$.\end{prop}
\begin{proof} By \cite[Th.1]{kambayashi}, $H^1(G,\Z^{*k})$ is the amalgamed sum (or pushout) of $k$ copies of $H^1(G,\Z)$ in the category of pointed sets. But, if $(A,a)$ and $(B,b)$ are two pointed sets, then their amalgamed sum is $(C,(a,b))$, where $C=(A\times\{b\})\cup(\{a\}\times B)$. Thus, Theorem \ref{coho} allows us to conclude that $H^1(G,C)$ is finite independently of the action of $G$.(\footnote{In fact, one can prove directly owing to the definitions that $H^1(G,\Z)$ is finite independently of the action of $G$, without using the deep Theorem \ref{coho}.})\end{proof}
In the second case of this alternative, Theorem \ref{coho} shows that $H^1(G,A)$ is finite. This allows us to study some examples of surfaces with automorphisms of positive entropy :
\begin{prop} The surfaces constructed in \cite[Th.2]{bedford-kim} and in \cite[7.1,7.2]{mc-mullen} have a finite number of equivalence classes of real structures. \end{prop}
\begin{proof} In \cite[Th.2]{bedford-kim}, it is proved that the automorphism group of their surfaces $X$ satisfies $\Aut^*X \simeq \Z/2\Z*\Z/2\Z\simeq\Z\rtimes \Z/2\Z$ and $\Aut^{\#}X$ is finite, and in \cite[7.1,7.2]{mc-mullen}, Mc Mullen obtains an automorphism group isomorphic to $\Z\rtimes T$, where $T$ is a finite group. In any case, Theorem \ref{coho} allows us to conclude. \end{proof}
\indent\textbf{Acknowledgements.} The author is grateful to Frédéric Mangolte for asking him this question, and also for his advice. We want to thank Jérémy Blanc, Serge Cantat, Stéphane Druel, Viatcheslav Kharlamov and Stéphane Lamy for useful discussions. 

\end{document}